\documentclass[12pt]{amsart}
\usepackage{amsfonts}
\usepackage{graphicx}
\usepackage{tabularx}
\usepackage{array}
\usepackage[usenames,dvipsnames]{color}
\usepackage{comment}
\usepackage{amsmath}
\usepackage{amsthm}
\usepackage{amssymb}
\usepackage{algpseudocode}
\usepackage{algorithm}
\usepackage{fullpage}
\usepackage{multirow}
\usepackage[dvipsnames]{xcolor}
\usepackage{listings}

\usepackage{soul}

\def\N{\mathbb{N}}

\DeclareMathOperator{\ch}{ch}

\newtheorem{theorem}{Theorem}[section]
\newtheorem{proposition}[theorem]{Proposition}

\newtheorem{question}[theorem]{Question}
\newtheorem{claim}[theorem]{Claim}
\newtheorem{lemma}[theorem]{Lemma}
\newtheorem{corollary}[theorem]{Corollary}
\theoremstyle{definition}
\newtheorem{definition}[theorem]{Definition}
\newtheorem{example}[theorem]{Example}

\algnewcommand\algorithmicallof{\textbf{all of}}
\algdef{S}[FOR]{AllOf}[1]{\algorithmicallof\ #1\ \textbf{with condition}}
\algdef{S}[FOR]{AllOfAssign}[2]{#2 $\gets$ \algorithmicallof\ #1\ \textbf{with condition}}
\algdef{E}[FOR]{EndAllOf}{\textbf{end all of}}

\algnewcommand\algorithmicanyof{\textbf{any of}}
\algdef{S}[FOR]{AnyOf}[1]{\algorithmicanyof\ #1\ \textbf{with condition}}
\algdef{S}[FOR]{AnyOfAssign}[2]{#2 $\gets$ \algorithmicanyof\ #1\ \textbf{with condition}}
\algdef{E}[FOR]{EndAnyOf}{\textbf{end any of}}

\author{Peter Bradshaw}
\address{Department of Mathematics, University of Illinois Urbana-Champaign, Urbana, IL, USA}
\email{pb38@illinois.edu}
\author{Tianyue Cao}
\address{Department of Mathematics, University of Illinois Urbana-Champaign, Urbana, IL, USA}
\email{tc37@illinois.edu}
\author{Atlas Chen}
\address{Department of Mathematics, University of Illinois Urbana-Champaign, Urbana, IL, USA}
\email{liguoc2@illinois.edu}
\author{Braden Dean}
\address{Department of Mathematics, University of Illinois Urbana-Champaign, Urbana, IL, USA}
\email{bdean9@illinois.edu}
\author{Siyu Gan}
\address{Department of Mathematics, University of Illinois Urbana-Champaign, Urbana, IL, USA}
\email{siyugan3@illinois.edu}
\author{Ram\'on I.~Garc\'ia}
\address{Department of Mathematics, University of Illinois Urbana-Champaign, Urbana, IL, USA}
\email{rig2@illinois.edu}
\author{Amit Krishnaiyer}
\address{Department of Mathematics, University of Illinois Urbana-Champaign, Urbana, IL, USA}
\email{amitlk2@illinois.edu}
\author{Grace McCourt}
\address{Department of Mathematics, Iowa State University, Ames, IA, USA}
\email{gmccourt@iastate.edu}
\author{Arvind Murty}
\address{Department of Mathematics, University of Illinois Urbana-Champaign, Urbana, IL, USA}
\email{amurty2@illinois.edu}

\thanks{This paper originates from an Illinois Math Lab undergraduate project at University of Illinois Urbana-Champaign. Peter Bradshaw received support from NSF RTG grant DMS-1937241.}

\def\epsilon{\varepsilon}

\begin{document}
\title{Chip games and multipartite graph paintability}
\maketitle
\begin{abstract}
    We study the paintability, an on-line version of choosability, of complete multipartite graphs. We do this by considering an equivalent chip game introduced by Duraj, Gutowski, and Kozik \cite{duraj2015chip}. We consider complete multipartite graphs with $ n $ parts of size at most 3. Using a computational approach, 
    we establish upper bounds on the
    paintability of such graphs for small values of $ n. $ 
    
    The choosability of complete multipartite graphs is closely related to value $ p(n, m) $, the minimum number of edges in a $n$-uniform hypergraph with no panchromatic $m$-coloring.
    We consider an online variant of this parameter $ p_{OL}(n, m), $ introduced by Khuzieva et al.~\cite{khuzieva2017} using a \emph{symmetric chip game}. With this symmetric chip game, we find an improved upper bound for $ p_{OL}(n, m)$ when $m \geq 3$ and $n$ is large. Our method also implies a lower bound on the paintability of complete multipartite graphs with $m \geq 3$ parts of equal size.
\end{abstract}

\section{Background}
\subsection{Introduction to paintability}
A \emph{proper coloring} of a graph $G$ 
is a coloring $f:V(G) \rightarrow \mathbb N$
of the vertex set of
$G$ such that no two adjacent vertices have the same color. If there exists a proper coloring of $ G $ with $ r $ colors, then we say that $ G $ is \emph{$r$-colorable}.
The \emph{chromatic number} of $ G $ (denoted by $ \chi(G) $) is the minimum number $ r $ such that $ G $ is $ r $-colorable.

Now, consider a function $L : V(G) \rightarrow 2^\N $ (called a \emph{list-assignment}) that gives each vertex $v$ its own set $L(v) \subseteq \mathbb N$ of possible colors. An \emph{$ L $-coloring of $ G $} is a proper coloring 
$f:V(G) \rightarrow \mathbb N$ such that $f(v) \in L(v)$ for each $v \in V(G)$.
We say that a graph $ G $ is
\emph{$L$-colorable} if $ G $ can be properly colored with the list-assignment $ L: V(G) \rightarrow \N. $ If $ G $ is $ L $-colorable for every list-assignment $ L: V(G) \rightarrow 2^\N $ satisfying $ |L(v)| \geq r $ for every vertex $ v \in V(G) $,
then we say that $ G $ is \emph{$ r $-choosable}. 
The  \emph{choosability} of $G$, denoted by $\ch(G)$, is the minimum number $ r $ such that $ G $ is $ r $-choosable.

Although the idea of choosability was established in the 1970s independently by Vizing \cite{Vizing} and Erd\H{o}s, Rubin, and Taylor \cite{erdos1979choosability}, it was only in 2009 that Zhu \cite{zhu2009line} and Schauz \cite{Schauz} independently 
introduced the on-line setting of list coloring.
In this setting, the \emph{on-line list coloring game} is played on undirected graphs, where two players, typically referred to as Lister and Painter, compete under specific rules involving the coloring of graph vertices with a set of available colors. The game consists of the following steps:

\begin{itemize}
    \item \emph{Preparation:} An undirected simple graph $G$ and threshold integer $ r \geq 1$ is given.
    \item \emph{Lister Turn:} On Turn $i$,
    Lister chooses a vertex set $S \subseteq V(G)$
    consisting of currently unpainted vertices  and presents the vertices in $S$ to Painter. 
    \item \emph{Painter Turn:} On Turn $i$, Painter chooses an independent set $I \subseteq S$ and paints all of the vertices in $I$ with the color $i$.
    \item \emph{Turn Alternation:}  The game proceeds in rounds.
    A \emph{round} consists of Lister's turn followed by Painter's turn.
    \item \emph{Win Conditions:} At the end of a round, Painter wins if all of the vertices have been painted, and Lister wins if there exists an unpainted vertex that has been presented by the Lister at least $ r $ times.
\end{itemize}
If the Painter has a winning strategy on $G$ with a threshold of $ r, $ then $ G $ is \emph{$r$-paintable}. The \emph{paintability} of $ G $ (denoted by $ \chi_P(G) $) is the minimum integer $ r $ such that $ G $ is $ r $-paintable.

One may observe that for any given graph $G$, we have \[\chi(G) \leq \ch(G) \leq \chi_P(G).\]
To see the first inequality, note that if $ G $ is $ r $-choosable, then  $G$ can be properly colored with the list-assignment $ L $ assigning $ \{1, \ldots, r\} $ to each vertex, meaning it can simply be properly colored with $r$ colors.

For the second inequality, suppose that $G$ is not $k$-choosable, so that $\ch(G) > k$. Then, there exists a $k$-assignment $L:V(G) \rightarrow 2^{\mathbb N}$ for which $G$ has no $L$-coloring. We label the colors of $\bigcup_{v \in V(G)} L(v)$ as $1, \dots, \ell$. Then, on each turn $i$, Lister reveals the color $i$ at each $v \in V(G)$ for which $i \in L(v)$. If Painter wins the game, then at the end of the game, $G$ has an $L$-coloring, a contradiction. Therefore, Painter has no winning strategy, implying $\chi_P(G) > k$.

In 2015, Duraj, Gutowski and Kozik  \cite{duraj2015chip}
modified a chip game of Aslam and Dhagat \cite{AslamDhagat} 
to be
equivalent to 
the on-line list coloring game played on \emph{complete multipartite graphs}. Before defining this chip game, we introduce the following concept: 
A \emph{complete k-partite graph}, denoted as $K_{n_1, \ldots, n_k}, $ is a graph whose vertices can be partitioned into $k$ different independent sets $ A_i$, called \emph{parts}, with $ |A_i| = n_i$, such that for any pair of vertices $(u,v)$ that belongs to two different parts $A_i$, an edge connects $u$ and $v$. 
When $k$ is not specified, such graphs $G$ are called complete multipartite graphs.
If the $k$ parts of $ G $ all have cardinality $ m, $ then we write write $ G = K_{m \star k}. $

Intuitively, in order to properly color a complete $k$-partite graph with as few colors as possible, each part needs to be colored with a unique color that is not used
in the other parts. Thus, for every
complete $k$-partite graph with $k$ nonempty parts, $\chi(G) = k$.

\subsection{Paintability: Known results}
In the study of choosability and paintability, complete multipartite graphs have received special attention due to their connection with Ohba's conjecture \cite{Ohba}. Ohba's original conjecture from 2002 asserts that if $G$ is a graph satisfying $|V(G)| \leq 2 \chi(G) + 1$, then $\chi(G) = \ch(G)$. Later, in 2012, Huang, Wong, and Zhu \cite{huang2012application} posed the on-line Ohba's conjecture, which asserts that if $|V(G)| \leq 2 \chi(G)$, then $\chi(G) = \chi_P(G)$.
The stronger assumption that $|V(G)| \leq 2 \chi(G)$ is necessary, as Kim, Kwon, Liu, and Zhu \cite{KKLZ} proved
that 
\begin{equation}
\label{eqn:one-3}
\chi_P(K_{2 \star (n - 1), 3 \star 1}) = n + 1  > n = 
\chi(K_{2 \star (n - 1), 3 \star 1}),
\end{equation}
and $K_{2 \star (n - 1), 3 \star 1}$ is a graph on $2n+1$ vertices.
As a proof of either version of Ohba's conjecture for complete multipartite graphs implies the conjecture for all graphs, complete multipartite graphs have become a special focus for research on choosability and paintability.


For certain classes of complete multipartite graphs, exact values for choosability and paintability are known.
In 1993, Alon  \cite{alon1993restricted} proved that for each integer $n \geq 1$,
$ch(K_{2 \star n}) = n$.
Later Kim, Kwon, Liu and Zhu \cite{KKLZ} showed
that 
$\ch(K_{2\star n}) = \chi_P(K_{2 \star n}) = n$.
The choosability for $K_{3 \star n}$ was later given by Kierstead  \cite{kierstead2000choosability}, who showed that
$ch(K_{3 \star n})=\left \lceil \frac{4n - 1}{3} \right \rceil$.
Regarding the paintability of $K_{3 \star n}$, Kozik, Micek and Zhu \cite{KMZ} found an upper bound for these graphs, implying the following:
\begin{equation}
\label{eqn:K3n}
\left \lceil \frac{4n - 1}{3} \right \rceil \leq ch(K_{3 \star n}) \leq  \chi_P(K_{3 \star n}) \leq \frac{3n}{2}.   
\end{equation}

As the choosability 
of a graph
is always at most its paintability, 
it is natural then to investigate the gap between the two parameters. When Zhu \cite{zhu2009line} introduced 
online list coloring,
he showed that $ K_{6,9} $ and $ K_{6,10} $ are not 3-paintable, but they are known to be 3-choosable. He then asked if the gap between 
the choosability and paintability of a single graph could be made arbitrarily large.
Given the gap between the lower and upper bound of $\chi_{P}(K_{3 \star n})$, Kozik, Micek, and Zhu \cite{KMZ} considered the graphs $K_{3 \star n}$ as natural candidates 
for study, and they
asked the following question:
\begin{question}
    What is the relationship between $\ch(K_{3 \star n})$ and $\chi_P(\chi_{3 \star n})$ as $n \rightarrow \infty$? In particular, is $\chi_P(K_{3 \star n}) - \ch(K_{3\star n})$ unbounded as $n \rightarrow \infty$?
\end{question}

For the question of whether the gap between a graph's choosability and paintability can be arbitrarily large, the graph $K_{n,n}$ ended up being the key to the answer of this question. 
Duraj, Gutowski and Kozik  \cite{duraj2015chip}
pointed out that a result of Radhakrishnan and Srinivasan \cite{radhakrishnan} for proper hypergraph $2$-coloring implies the following upper bound:
$$ch(K_{n,n}) = \log_2 n - \Omega(\log_2 \log_2 n) \text{ as } n \rightarrow \infty.$$
On the other hand, letting  $ N $ be the smallest number such that $ K_{N,N} $ is not $ k $-paintable, Aslam and Dhagat \cite{AslamDhagat} proved that $ 2^{k-1} \leq N $ and also provided a constructive strategy which demonstrated that $ N \leq k \phi^{2k} $, where $\phi = \frac{\sqrt 5 + 1}{2}$ is the golden ratio.
Later, Duraj, Gutowski and Kozik 
\cite{duraj2015chip} produced a strategy that improved the upper bound to $ N \leq 2^{k+3}. $ As such,
$N = \Theta(2^k),$
which implies that 
$$\chi_P(K_{n,n}) = \log_2 n + O(1) \text{ as } n \rightarrow \infty.$$
This proves that $\chi_P(K_{n,n}) - \ch(K_{n,n}) \rightarrow \infty$ as $n \rightarrow \infty$, showing that
the gap between the choosability and the paintability of a graph can be arbitrarily large.


\subsection{Chip games}

A $ (k, n_1, \ldots, n_m) $ chip game is equivalent to the On-line List Coloring Game  played on a complete multipartite graph $K_{n_1, \dots, n_m}$.
In this game, two players, typically referred to as Pusher and Remover, compete under specific rules. The game is defined as follows.

\begin{definition}
    A $ (k, n_1, \ldots, n_m) $ chip game consists of the following steps.

\begin{itemize}
    \item \emph{Preparation:} A table with $m$ columns and rows labeled $0, \dots, k$ is provided as a game board.
    For $i \in\{1, \dots, m\}$, $n_i$ identical chips are placed in row $0$ of the $i$th column.
    \item \emph{Pusher Turn:} Pusher chooses a nonempty set $S$ of chips on the board, and then pushes  each chip in the set $S$ exactly $1$ row upwards, so that a chip in $S$ occupying row $i$ moves to row $i+1$.
    \item \emph{Remover Turn:} Remover chooses exactly one column $C$ and removes all chips belonging to $S$ in the column $C$.
    \item \emph{Turn Alternation:} The \{Pusher, Remover\} turn pair is referred to as a Round. The turn then alternates between Pusher and Remover for an arbitrary number of Rounds.
    \item \emph{Win Condition:} Pusher wins if at least $1$ chip occupies row $k$ at the end of a round, and Lister wins if all chips are removed from the board before Pusher's win condition is reached.
\end{itemize}
\end{definition}
We note that on each turn of a $(k,n_1, \dots, n_m)$ chip game, Pusher is required to move a nonempty set of chips. Furthermore, the game ends if either a single chip moves $k$ times or if all chips are removed. Therefore, the $(k,n_1, \dots, n_m)$ chip game ends after fewer than $k(n_1 + \dots + n_m)$ turns.
We abbreviate the $(k, \underbrace{n, \dots, n}_{m \textrm{ times}})$ chip game as the $(k, n \star m)$ chip game. 

 Duraj, Gutowski, and Kozik \cite{duraj2015chip} observed the following theorem for $m=2$, but their ideas naturally generalize to all positive $m$. We include a proof for completeness.
\begin{theorem}
\label{thm:chippaint}
    $\chi_P(K_{n_1, \dots, n_m}) > k$ if and only if Pusher has a winning strategy in the $(k, n_1, \dots, n_m)$ chip game.
\end{theorem}
\begin{proof}
    Write $G = K_{n_1, \dots, n_m}$. Write $A_1, \dots, A_m$ for the $m$ parts of $G$. 
    For each part $i \in \{1, \dots, m\}$, we write $c^i_1, \dots, c^i_{n_i}$ for the chips in column $i$. Similarly, we write $v^i_1, \dots, v^i_{n_i}$ for the $n_i$ vertices in part $A_i$ of $G$.

    First, suppose that $\chi_P(G) > k$. Then, Lister has a winning strategy $\Sigma$ in the online list coloring game on $G$ with $k$ colors. Pusher proceeds as follows. 
    Pusher first initializes an instance of the online list coloring game on $G$.
    On each Round $i$ of the chip game, Pusher first uses $\Sigma$ to calculate a winning move of Lister on Round $i$ of the online list coloring game on $G$. This winning move of Lister is in the form of a set $S \subseteq V(G)$. Then, for each $v^i_j \in S$, Pusher pushes the chip $c^i_j$. Next, if Remover removes Column $i$, then in the online list coloring game on $G$, Pusher imagines that Painter colors every vertex in $S \cap A_i$. It is straightforward to check that after each round of the chip game, a chip $c^i_j$ remains on the board if and only if $v^i_j$ is uncolored, and $c^i_j$ occupies row $r$ if and only if Lister has presented a color at $v^i_j$ exactly $r$ times. As the strategy $\Sigma$ allows Lister to present a color $k$ times at some vertex $v^i_j$ so that the vertex $v^i_j$ is not colored, it thereby follows that Pusher succeeds in pushing a chip $c^i_j$ $k$ times so that $c^i_j$ is not removed. Therefore, Pusher has a winning strategy in the $(k,n_1, \dots, n_m)$ chip game.

    For the other direction, the result follows from the same correspondence described above.
\end{proof}

\subsection{Our results}

In Section \ref{sec:comp}, we compute the paintability of some complete multipartite graphs, including $K_{3 \star n}$ for $n \in \{4,5,6\}$,  as well as some unbalanced graphs such as $K_{2\star 4,3 \star 3}$, with the help of a computer.
For $K_{3 \star n}$ and $n \in \{4,5,6\}$, we show that the paintabilities are $\lceil \frac{4n-1}{3}\rceil$, which agrees with their choosabilities. 
In Section \ref{sec:hyper}, we introduce the \emph{online hypergraph pancoloring} problem, 
which was observed by Akhmejanova, Bogdanov, and Chelnokov to be representable using a chip game.
We also adapt a chip game strategy of Duraj, Gutowski, and Kozik \cite{duraj2015chip}
to obtain an upper bound for the minimum number of edges in a hypergraph that cannot be panchromatically $r$-colored online, improving a previous result of Khuzieva et al.~\cite{khuzieva2017}.
Finally, in Section 
\ref{sec:conclusion}, 
we pose some questions.

\section{Computational Approach}
\label{sec:comp}

In this section, we aim to
use a computational approach to
accurately determine the paintabilities for several complete multipartite graphs for which there is a gap between the known lower and upper bounds of their paintabilities. As indicated in Equation (\ref{eqn:K3n}), the exact paintabilities of many complete multipartite graphs $K_{3 \star n}$ remain undetermined, as do the paintabilities of many graphs composed of partite sets consisting of $2$ and $3$ vertices. 
We aim to determine the paintabilities of certain complete multipartite graphs of this form by computationally evaluating chip games with $2$ or $3$ chips per column.
%

To investigate the paintabilities of complete multipartite graphs, we use the chip game of Duraj, Gutowski, and Kozik \cite{duraj2015chip}, along with Theorem \ref{thm:chippaint}. In order to determine whether $\chi_P(K_{n_1,\dots,n_m}) > k$ for some positive integer $k$, we aim to determine whether Pusher has a winning strategy in the $(k,n_1, \dots, n_m)$ chip game. Therefore, in this section, we outline a computational procedure
that determines whether a given position in the chip game is winning or losing for Pusher, so that we can apply this algorithm to the initial position of the $(k,n_1, \dots, n_m)$ chip game.

\subsection{Formal framework}
We describe a formal framework in which our computational procedure operates. We begin with some definitions.
\begin{definition}
A \emph{column state $C$} with $K$ chips is a sequence of pairs
    \[C = ((r_1, m_1), (r_2, m_2), \cdots, (r_K, m_K)),\]
    where $ r_1 \geq r_2 \geq \cdots \geq r_K \geq -1$ and $m_i \in \{0,1\}$ for each $i \in \{1, \dots, K\}$.
    Each pair $(r_i, m_i)$ represents a chip $A_i$.
    If $r_i \geq 0$, then
    $r_i$ represents the row of $A_i$; if $r_i = -1$, then $A_i$ has been removed from the board.
    We also let $m_i = 1$
    if it is Remover's turn and $A_i$ has just been pushed (and thus is eligible for removal); otherwise, $m_i = 0$.
\end{definition}

\begin{example}
    The column state $C = ((5, 0), (3, 1), (0, 0), (-1, 0), (-1, 0))$ represents a column with $5$ chips $A_1, \dots, A_5$. Chips $A_1, A_2, A_3$ in are in rows $5, 3, 0$ respectively, and two chips $A_4, A_5$ removed from the board. Each removed chip is represented by an ordered pair with first entry $-1$.
    It is Remover's turn, and chip $A_3$ has just been pushed and is eligible for removal.
    The column state is shown in Figure \ref{fig:column}.
\end{example}

\begin{figure}
    \begin{center}
        \begin{tabular}{|c|c|}
            \hline
            5 & $A_1$ \\ \hline
            4 &  \\ \hline
            3 & $A_2$ \\ \hline
            2 &  \\ \hline
            1 &  \\ \hline
            0 & $A_3$ \\ \hline
        \end{tabular}
    \end{center}
    \caption{The column in the figure has chips in rows $5$, $3$, and $0$. There are two additional chips that were originally in the column but were removed. The current state of this column is represented by the sequence of pairs $((5,0), (3,1), (0,0), (-1,0), (-1,0))$.}
    \label{fig:column}
\end{figure}

\begin{definition}
     A \emph{board} $B$ is a set of $N$ column states with $K$ chips:
    \[
        B = \{ C_1, C_2, \cdots, C_N \}.
    \]
In a board $B$, each element $C_i$ corresponds to a column in a chip game. Therefore, a board with $N$ elements $C_i$ 
uniquely corresponds to an arrangement of chips in a chip game with $N$ columns.
\end{definition} 
\begin{definition}
    A \emph{game state} is a triple $G = (B, \Gamma, P)$ where $B$ is the board, $\Gamma \in \mathbb N$ is the \emph{winning threshold}, and $P \in \{ \text{Pusher}, \text{Remover} \}$ is the player to move.
\end{definition}
    The value $\Gamma$ represents the row that a chip needs to reach in order for Pusher to win the chip game. We elaborate on the precise meaning of $\Gamma$ in Definition \ref{def:winning-losing-states}.
We note that game state refers to a snapshot of the chip game, with information including the positions of all chips, the player to make the move, and the number of rows.
We note further that a $(\Gamma, K \star N)$ chip game has an initial game state $(B, \Gamma, \text{Pusher})$ where $B = \{ C_1, C_2, \cdots, C_N \}$ and $C_i = (\underbrace{(0, 0), (0, 0), \cdots, (0, 0)}_{K \textrm{ times}})$ for each $i \in \{1, \dots, N\}$.

\begin{definition}
    $\mathcal{G}(N, K, \Gamma)$ denotes the set of all game states with $N$ columns and $K$ chips per column. $\mathcal{G}_p(N, K, \Gamma) \subseteq \mathcal{G}(N, K, \Gamma)$ denotes the set of game states where it is Pusher's move, and $\mathcal{G}_r(N, K, \Gamma) \subseteq \mathcal{G}(N, K, \Gamma)$ denotes the set of game states where it is Remover's move.
\end{definition}

\begin{definition}
    \label{def:winning-losing-states}
    $\mathcal{W}(N, K, \Gamma) \subseteq \mathcal{G}_p(N, K, \Gamma)$ denotes the set of game states where Pusher's winning condition is met, that is, there is a chip at or above row $\Gamma$. The set $\mathcal{L}(N, K, \Gamma) \subseteq \mathcal{G}_p(N, K, \Gamma)$ denotes the set of game states where Pusher's losing condition is met, that is,
    there are no chips left on the board.
\end{definition}

Note that in the definitions above, we look
at the game from Pusher's perspective, so that the term \emph{winning} refers to Pusher winning, and the term \emph{losing} refers to Pusher losing.
Note also that both $\mathcal{W}(N, K, \Gamma)$ and $\mathcal{L}(N, K, \Gamma)$ are subsets of $\mathcal{G}_p(N, K, \Gamma)$, which means that only states with Pusher to move are considered as winning or losing.

\begin{definition}
\label{def:move}
    Given $N$, $K$, and $\Gamma$, a \emph{Pusher move} $\sigma_p \in \{ 0, 1 \}^{NK}$ 
    on a board state $G \in \mathcal{G}(N, K, \Gamma)$ is a binary string of length $NK$ whose entries are indexed by the $NK$ chips,
    such that the $i$th bit is $1$ if and only if Pusher pushes the $i$th chip.
    A \emph{Remover move} $1 \leq \sigma_r \leq N$ is an integer representing the column chosen by Remover.
\end{definition}
    For a Pusher/Remover move $\sigma$, we use $\sigma(G)$ to denote the resulting game state after applying the move.
    We note that according to Definition \ref{def:move}, Pusher is allowed to push chips that have already been removed from the board. This allowance does not affect the game, as a chip never returns to the board after being removed.

We use the following framework, described by Knuth and Moore \cite{KnuthMoore}, to assign a numerical value $F(G)$ to each game state $G \in \mathcal G(N,K,\Gamma)$.
First, for each $G \in \mathcal W(N,K,\Gamma)$, we let $F(G) = 1$, and for each $G \in \mathcal L(N,K,\Gamma)$, we let $F(G) = 0$. Then, we extend $F$ to all of $\mathcal G(N,K,\Gamma)$ by defining $F(G)$ as follows for each $G \in \mathcal G(N,K,\Gamma) \setminus (\mathcal W(N,K,\Gamma) \cup \mathcal L(N,K,\Gamma) )$:
\[F(G) = 
\begin{cases}
    \max_{\sigma} F(\sigma(G)) & \textrm{ if it is Pusher's move in $G$} \\
    \min_{\sigma} F(\sigma(G)) & \textrm{ if it is Remover's move in $G$},
\end{cases}
\]
where $\sigma$ runs over all legal moves of the player with the move in the position $G$. Knuth and Moore \cite{KnuthMoore} show that for finite games, a function $F$ of this kind is well defined.
We say that a state $G \in \mathcal G(N,K,\Gamma)$ is \emph{winning} if $F(G) = 1$, and we say that $G$ is \emph{losing} if $F(G) = 0$.

\begin{definition}
    Define the \emph{partial order $\geq$} on column states in $\mathcal G_p(N,K,\Gamma)$
    as follow: If 
    \[C = ((r_1, m_1), (r_2, m_2), \cdots, (r_K, m_K))\] and 
    \[C' = ((r_1', m_1'), (r_2', m_2'), \cdots, (r_K', m_K')),\]
    then $C \geq C'$ if and only if for all $i$ from $1$ to $K$, $r_i \geq r_i'$.

\end{definition}

\begin{definition}
    Given $N$, $K$, and $\Gamma$, define the \emph{partial order $\geq$} on boards $B = \{ C_1, C_2, \cdots, C_N \}$ and $B' = \{ C_1', C_2', \cdots, C_N' \}$ as follows: $B \geq B'$ if and only if there exists a permutation $\pi$ of $\{ 1, 2, \cdots, N \}$ such that for all $1 \leq i \leq N$, $C_i \geq C_{\pi(i)}$.

\end{definition}

\begin{definition}
    Given $N$, $K$, and $\Gamma$, define the \emph{partial order $\geq$} on game states $G = (B, \Gamma, P) \in \mathcal{G}_p(N, K, \Gamma)$ and $G' = (B', \Gamma, P') \in \mathcal{G}_p(N, K, \Gamma)$: $G \geq G'$ if and only if $B \geq B'$.
\end{definition}


The following proposition follows easily from an inductive argument using the recursive definition of $F$.

\begin{proposition}
        \label{prop:comp}
    Given $N$, $K$, and $\Gamma$, let $G, G' \in \mathcal{G}_p(N, K, \Gamma)$. If $G \geq G'$, then $F(G) \geq F(G')$. In other words, if $G \geq G'$, then:
    \begin{itemize}
        \item If $G'$ is winning state, then $G$ is winning state.
        \item If $G$ is losing state, then $G'$ is losing state.
    \end{itemize}
\end{proposition}

\begin{definition}
    Given $N$, $K$, and $\Gamma$, a \emph{winning closure} $C_p \subseteq \mathcal{G}_p(N, K, \Gamma)$ is a set of game states such that for any game state $G \in C_p \setminus \mathcal{W}(N, K, \Gamma)$, there exists a Pusher move $\sigma_p$ such that for any Remover move $\sigma_r$, $\sigma_r(\sigma_p(G)) \geq G'$ for some $G' \in C_p \cup \mathcal{W}(N, K, \Gamma)$.

    Similarly, a \emph{losing closure} $C_r \subseteq \mathcal{G}_p(N, K, \Gamma)$ is a set of game states such that for any game state $G \in C_r \setminus \mathcal{L}(N, K, \Gamma)$ and Pusher move $\sigma_p$, there exists a Remover move $\sigma_r$ such that $\sigma_r(\sigma_p(G)) \leq G'$ for some $G' \in C_r \cup \mathcal{L}(N, K, \Gamma)$.
\end{definition}

We observe that by the definition of $F$,
the set of game states $G \in \mathcal G_p(N,K,\Gamma)$ for which $F(G) = 1$ form a winning closure, and the set of game states  for which $F(G) = 0$ form a losing closure. This gives us the following proposition.

\begin{proposition}
    \label{thm:win-lose-closure}
    A game state $G \in \mathcal G_p(N,K,\Gamma)$ is a winning state if and only if $G \in C_p$ for some winning closure $C_p \subseteq \mathcal G_p(N,K,\Gamma)$. Similarly, $G$ is a losing state if and only if $G \in C_r$ for some losing closure $C_r \subseteq \mathcal G_p(N,K,\Gamma)$.
\end{proposition}

\subsection{Algorithmic techniques}
In this section, we describe an algorithm that solves the following problem: given $N, K, \Gamma$, let $G_0 = (B, \Gamma, \text{Pusher})$ be the starting game state where
$B$ is a board consisting of 
$N$ column states, each with $K$ pairs of $(0, 0)$. The board $B$ represents the starting configuration of the $(\Gamma, K \star N)$ chip game.
Our goal is to check whether $G_0$ is a winning state (i.e. $F(G_0) = 1$). 
By Theorem \ref{thm:chippaint},
the paintability $\chi_P(K_{N \star K})$ is the smallest $\Gamma$ for which the initial state $G_0$ of the $(\Gamma, K \star N)$ chip game satisfies $F(G_0) = 0$.
By computing $F(G_0)$ for appropriate values of $\Gamma$,
we can find the exact value of the paintability of $K_{N \star K}$ for some fixed $N$ and $K$.

The recursive definition of the function $F$ allows us to compute the value of $F(G_0)$ by defining a tree of board states in which each leaf contains a state in which the game has ended, and then working backwards to compute the value of each intermediate game state.
We use the term \emph{minimax algorithm} to refer to an algorithm that computes $F(G_0)$ using the recursive definition of $F$.
However, a naive minimax algorithm needs to consider every possible game state that can be reached from the initial position and therefore is practically infeasible. 
Therefore, we develop several tools that help us construct a more efficient minimax algorithm that evaluates $F(G_0)$.

\subsubsection{Game state comparison}
Given two game states $G,G' \in \mathcal G_p(N,K,\Gamma)$, Proposition \ref{prop:comp} implies that if $G \geq G'$, then $F(G) \geq F(G')$.
Therefore, if we know that $F(G') = 1$, then we can conclude that $F(G) = 1$; similarly, if we know that $F(G) = 0$, then we can conclude that $F(G') = 0$. 
Below, we describe an efficient procedure for checking whether $G \geq G'$ for two game states $G,G' \in \mathcal G_p(N,K,\Gamma)$.

 Given two game states $G,G' \in \mathcal G_p(N,K,\Gamma)$ with respective boards $B$ and $B'$, our goal is to check whether $B \geq B'$ in polynomial time. To begin, write
\[
    B = \{ C_1, C_2, \cdots, C_N \},
\]
\[
    B' = \{ C_1', C_2', \cdots, C_N' \}.
\]
First, we construct a \emph{helper graph} $H(B,B')$. We 
let $H(B, B')$ be a simple undirected unweighted helper graph with vertices
\[
    V(H) = \{ V_1, V_2, \cdots, V_N, V_1', V_2', \cdots, V_N' \}
\]
The vertices are named so that each vertex corresponds to a column state in either $B$ or $B'$. The edges are defined as
\[
    (V_i, V_j') \in E(H) \text{ if and only if } C_i \geq C_j'.
\]
Note that 
constructing $H(B,B')$ takes $O(N^2K)$ time.

We observe that 
    $H$ is a bipartite graph with partite sets $\{ V_1, V_2, \cdots, V_n \}$ and $\{ V_1', V_2', \cdots, V_n' \}$.
We also observe that by definition, $B \geq B'$ if and only if $H(B,B')$ has a perfect matching. Therefore, to determine whether $B \geq B'$, we apply the Hopcroft-Karp algorithm \cite{HK} to $H(B,B')$, which finds a maximum matching in $H(B,B')$ in 
 $O(|E|\sqrt{|V|}) = O(N^{2.5})$ time.
 In particular, the algorithm determines whether $H(B,B')$ has a perfect matching.
Therefore, given $G,G' \in \mathcal G_p(N,K,\Gamma)$, we can check whether $G \geq G'$ in $O(N^{2.5})$ time.

\subsubsection{Pruning}
In this subsection, 
we explore the following question:
Given a game state $G$, 
which Pusher/Remover moves are ``redundant"? If we can show that some moves are worse than other moves, we can assume that Pusher/Remover will not make those moves, thereby reducing the number of nodes the tree of board states involved in computing the value $F(G_0)$ for the initial board state $G_0 \in \mathcal G_p(N,K,\Gamma)$.

Consider a game with state $G$ and Pusher being the next to move. We define \emph{better (worse) than} as follows:

\begin{definition}
    Given $N$, $K$, and $\Gamma$, let $G \in \mathcal{G}_p(N, K, \Gamma)$. Let $\sigma_p, \sigma_p' \in \{ 0, 1 \}^{NK}$
    be two distinct Pusher
    moves. We say $\sigma_p$ is \emph{worse than} $\sigma_p'$ (or $\sigma_p'$ is \emph{better than} $\sigma_p$) if
    $F(\sigma_p(G)) \leq F(\sigma'_p(G))$.
\end{definition}

In other words, if $\sigma_p$ is worse than $\sigma_p'$, Pusher will always prefer choosing $\sigma_p'$ over $\sigma_p$ since $\sigma_p$ cannot possibly achieve a better result than $\sigma_p'$. Note that the term better/worse is dependent on the current game state $G$. If $\sigma_p$ is worse than $\sigma_p'$ for one game state $G$, the same might not hold for another game state $G'$.
When considering a game state $G \in \mathcal G_p(N,K,\Gamma)$, if we know that a Pusher move $\sigma_p$ is worse than another Pusher move $\sigma_p'$, then we may assume that Pusher does not play $\sigma_p$.

The following lemma gives a sufficient condition for a Pusher move $\sigma_p$ to be worse than another move $\sigma_p'$. In particular, the lemma implies that if a game state $G \in \mathcal G_p(N,K,\Gamma)$ contains two identical columns $C_1$ and $C_2$, then a move $\sigma_p$ that pushes a chip in $C_1$ but no chip in $C_2$ is suboptimal. This lemma is particularly useful for pruning Pusher moves in early states of the chip game, when many columns are likely to be identical.

\begin{lemma}
\label{lem:same-columns}
    Suppose that a board state $G = (\{ C_1, C_2, C_3, \cdots, C_N \}, \Gamma, \text{Pusher} )$ has two identical columns $C_1 = C_2$. Let $\sigma_p$ and $\sigma_p'$ be two Pusher moves such that
    \[
        \sigma_p(G) = \{ \{ C_1', C_2', C_3', \cdots, C_N' \}, \Gamma, \text{Remover} \}
    \]
    \[
        \sigma_p'(G) = \{ \{ C_1', C_1', C_3', \cdots, C_N' \}, \Gamma, \text{Remover} \}
    \]
    (Note that $\sigma_p(G)$ and $\sigma'_p(G)$ differ in the second column.) If $C_1' \geq C_2'$, then $\sigma_p$ is worse than $\sigma_p'$.
\end{lemma}
\begin{proof}
    We prove the lemma by contradiction. Suppose $\sigma_p$ is not worse than $\sigma_p'$,
    so that $F(\sigma_p(G)) = 1$ and $F(\sigma'_p(G)) = 0$.
    Then,
    \begin{enumerate}
        \item For all Remover moves $\sigma_r$, $\sigma_r(\sigma_p(G))$ is a winning state;
        \item There exists a Remover move $\sigma_r$ such that $\sigma_r(\sigma_p'(G))$ is a losing state.
    \end{enumerate}

    Fix a Remover move $\sigma_r$ for which $\sigma_r(\sigma_p'(G))$ is a losing state. Since columns $1$ and $2$ are identical in $\sigma_p'(G)$, we can switch the labels of the two columns if Remover chooses to remove Column $2$. Therefore, we assume without loss of generality that $\sigma_r \neq 2$. After $\sigma_p'$ and $\sigma_r$, we have the game state
    \[
        \sigma_r(\sigma_p'(G)) = \{ C_1'', C_1', C_3'', \cdots, C_N'' \}.
    \]

    Next, we consider the state $\sigma_r(\sigma_p(G))$ for the same $\sigma_r$. The resulting game state is identical to $\sigma_r(\sigma_p'(G))$ except for the second column:
    \[
        \sigma_r(\sigma_p(G)) = \{ C_1'', C_2', C_3'', \cdots, C_N'' \}
    \]

    Since $C_1' \geq C_2'$, we have $\sigma_r(\sigma_p'(G)) \geq \sigma_r(\sigma_p(G))$. By Proposition \ref{prop:comp}, since $\sigma_r(\sigma_p'(G))$ is a losing state, $\sigma_r(\sigma_p(G))$ is also a losing state, contradiction. Therefore, $\sigma_p$ is worse than $\sigma_p'$.
\end{proof}

We use a similar approach to prune Remover moves. 
\begin{definition}
    Given $N$, $K$, and $\Gamma$, let $G \in \mathcal{G}_r(N, K, \Gamma)$. Let $\sigma_r, \sigma_r' \in \{ 1, 2, \cdots, N \}$ be two different moves. We say $\sigma_r$ is \emph{worse than} $\sigma_r'$ (or $\sigma_r'$ is \emph{better than} $\sigma_r$) if $F(\sigma_r(G)) \geq F(\sigma_r'(G))$.
\end{definition}
According to Proposition \ref{prop:comp}, 
if
    $\sigma_r(G) \geq \sigma_r'(G)$, then
    $\sigma_r$ is worse than  $\sigma_r'$.
Therefore, if $\sigma_r$ is worse than $\sigma_r'$, then we may assume without loss of generality that Remover does not play $\sigma_r$.

The algorithmic methods described above allow us to carry out the following procedure to determine whether the $(\Gamma, K \star N)$ chip game is winning or losing for Pusher, as follows. We aim to compute $F(B)$ for the initial state $B$ of the $(\Gamma, K \star N)$ chip game
using the recursive definition of $F$. When computing $F(B)$ recursively, 
we create and update a set of winning and losing states. If we find that some state $G \in \mathcal G_p(N,K,\Gamma)$ satisfies $G \geq G'$ for some winning state $G'$, then we can immediately conclude that $F(G) = 1$ without further computation. Similarly, if we find that some state $G \in \mathcal G_p(N,K,\Gamma)$ satisfies $G \leq G'$ for some losing state $G'$, then we can immediately conclude that $F(G') = 0$ without further computation. 
Also, given a state 
$G \in \mathcal G_p(N,K,\Gamma)$, we assume without loss of generality that Pusher does not play a move $\sigma_p$ that is worse than another move $\sigma_p'$. Similarly, given a state 
$G \in \mathcal G_r(N,K,\Gamma)$, we may assume without loss of generality that Remover does not play a move $\sigma_r$ that is worse than another move $\sigma_r'$.
In particular, when considering game states $G \in \mathcal G_p(N,K,\Gamma)$ for which multiple columns are identical, we use Lemma \ref{lem:same-columns}
to reduce the number of Pusher moves that we need to consider. Using these techniques, we find either a winning or losing closure containing the initial game state $G_0$ and thereby compute the value of $F(G_0)$.


\subsection{Verification Algorithm}
By using the algorithms outlined above, we can construct a more efficient minimax algorithm to determine whether the initial state $G_0$ of the 
 $(\Gamma, K \star N)$ chip game is winning or losing for Pusher.
 When our minimax algorithm computes $F(G_0)$, it in fact computes
a value $F(G)$ for many game states $G \in \mathcal G_p(N,K,\Gamma)$.
In particular, the algorithm outputs a winning close and a losing closure, and $G_0$ is contained in one of these closures.
By Theorem \ref{thm:win-lose-closure}, if $G_0$ is in a winning (losing) closure, then it is a winning (losing) state.
Therefore,
in order to verify the correctness of our minimax algorithm, we can use an independent verification algorithm to verify that the winning (losing)
closure containing $G_0$ is indeed a winning (losing) closure.

\textbf{Winning closure verification:}
The following algorithm verifies whether a given set $\mathcal S \subseteq \mathcal G_p(N,K,\Gamma)$ is a winning closure. In particular, if $\mathcal S$ is a winning closure containing the initial game state $G_0$, then the following algorithm can verify that Pusher wins the $(\Gamma, K \star N)$ chip game 
with optimal play. The algorithm is as follows.

We input a set $\mathcal S \subseteq \mathcal G_p(N,K,\Gamma)$ of game states. For each $G \in \mathcal S$, we execute the following procedure to determine whether $G$ is \emph{good} or \emph{bad}:
\begin{enumerate}
\item First, we check if some chip of $G$ is at or above row $\Gamma$; if so, then we say that $G$ is good.
\item Next, we search for a Pusher move $\sigma_p$ such that for all Remover responses $\sigma_r$ to $\sigma_p$,
$\sigma_r(\sigma_p(G)) \geq G'$ for some $G' \in \mathcal S$.
If such a Pusher move $\sigma_p$ exists, then we say $G$ is good. Otherwise, we say $G$ is bad.
\end{enumerate}
If every state $S \in \mathcal S$ is good, then we return \texttt{True}. If some state $S \in \mathcal S$ is bad, then we return \texttt{False}.

\textbf{Losing closure verification:}
The following algorithm verifies whether a given set $\mathcal S \subseteq \mathcal G_p(N,K,\Gamma)$ is a losing closure. In particular, if $\mathcal S$ is a losing closure containing the initial game state $G_0$, then the following algorithm can verify that Remover wins the $(\Gamma, K \star N)$ chip game 
with optimal play. The algorithm is as follows.

We input a set $\mathcal S  \subseteq \mathcal G_p(N,K,\Gamma)$  of game states. For each $G \in \mathcal S$, we execute the following procedure to determine whether $S$ is \emph{good} or \emph{bad}:
\begin{enumerate}
\item First, we check if some chip of $G$ is at or above row $\Gamma$; if so, then we say that $G$ is bad.
\item Next, we check whether all chips in $G$ are removed from the board. If all chips are removed from the board, then we say that $G$ is good.
\item For each Pusher move $\sigma_p$, we check whether Remover has a response move $\sigma_r$
such that $\sigma_r(\sigma_p(G)) \leq S'$ for some state $S' \in \mathcal S$.
If such a Remover move $\sigma_r$ exists for each Pusher move $\sigma_p$, then we say $G$ is good. Otherwise, we say $G$ is bad.
\end{enumerate}
If every state $S \in \mathcal S$ is good, then we return \texttt{True}. If some state $S \in C$ is bad, then we return \texttt{False}.

We note that the verification algorithms for checking whether a set $\mathcal S \subseteq \mathcal G_p(N,K,\Gamma)$ is a winning or losing closure run independently of the minimax algorithm used to produce $\mathcal S$. 
The verification algorithms also have the advantage of being much simpler than the algorithm that computes the value $F(G_0)$ for the initial state $G_0 \in \mathcal G(N,K,\Gamma)$. For the reader who is interested in verifying the correctness of our computational results, we recommend checking the verification algorithms due to their simplicity.

\subsection{Our Results}
By using our minimax algorithm to compute the value $F(G_0)$ for the initial state $G_0$ of various chip games, 
and by verifying this value $F(G_0)$ using our verification algorithms, 
we obtain the following results. The code for our algorithms is available at \texttt{https://github.com/Amitten77/IML-Paintability}.

\begin{theorem}
\label{thm:results}
The following hold:
\begin{itemize}
    \item $\chi_P(K_{3 \star 4}) = 5$.
    \item $\chi_P(K_{3 \star 5}) = \chi_P(K_{2 \star 3,3 \star 3}) = \chi_P(K_{2\star 2,3 \star 4}) = 7$.
    \item $\chi_P(K_{2 \star 1, 3 \star 5}) = \chi_P(K_{3 \star 6}) = \chi_P(K_{2 \star 4, 3 \star 3})  = \chi_P(K_{2 \star 3, 3 \star 4}) = 8$.
    \item $\chi_P(K_{2 \star 2, 3 \star 5}) = 9$.
\end{itemize}
\end{theorem}
\begin{proof}
    First, we consider the graph $K_{3 \star 4}$. Our algorithms show that the initial state of the $(3 \star 4, 4)$ chip game is winning for Pusher, but the initial state of the $(3 \star 4, 5)$ chip game is losing for Pusher. Therefore, by Theorem \ref{thm:chippaint}, $\chi_P(K_{3 \star 4}) = 5$. 

    The proofs for the other graphs listed in the theorem are similar.
\end{proof}
We summarize the results of Theorem \ref{thm:results} in Table \ref{table} and compare them to known lower and upper bounds of the paintabilities of the graphs considered in the theorem. The upper bounds in the table follow from 
Equation (\ref{eqn:K3n}),
as do the lower bounds for the paintabilities of graphs of the form $K_{3 \star n}$. The other lower bounds follow from Equation (\ref{eqn:one-3}).

\begin{table}
\begin{center}
\begin{tabular}{|c|c|c|c|}
\hline
\textbf{Graph} & \textbf{Known lower bound} & \textbf{Our computed value} & \textbf{Known upper bound} \\ \hline
$K_{3 \star 4}$ & 5 & 5 & 6 \\ \hline
$K_{3 \star 5}$ & 7 & 7 & 7 \\ \hline
$K_{2 \star 3, 3 \star 3}$ & 7 & 7 & 9 \\ \hline
$K_{2 \star 2, 3 \star 4}$ & 7 & 7 & 9 \\ \hline
$K_{2 \star 1, 3 \star 5}$ & 7 & 8 & 9 \\ \hline
$K_{3 \star 6}$ & 8 & 8 & 9 \\ \hline
$K_{2 \star 4, 3 \star 3}$ & 8 & 8 & 10 \\ \hline
$K_{2 \star 3, 3 \star 4}$ & 8 & 8 & 10 \\ \hline
$K_{2 \star 2, 3 \star 5}$ & 8 & 9 & 10 \\ \hline
\end{tabular}
\end{center}
\caption{The table shows the paintabilities of various complete multipartite graphs, compared with their previously known lower and upper bounds.}
\label{table}
\end{table}


\section{Online panchromatic hypergraph coloring}
\label{sec:hyper}
\subsection{Background}
We now pivot to hypergraph coloring, which was first shown to be closely related to the choosability of multipartite graphs by Erd\H{o}s, Rubin, and Taylor \cite{erdos1979choosability}.
    A \emph{hypergraph} $ H = (V, E) $ is a collection of vertices $ V $ and (hyper)edges $ E, $ where each hyperedge is a subset of $V$.
    A hypergraph is \emph{$ k $-uniform} if every hyperedge has cardinality $ k. $ In particular, a 2-uniform hypergraph is simply a graph.
%
    A \emph{panchromatic $ r $-coloring} of a hypergraph $ H = (V, E) $ is an $ r $-vertex coloring $ f : V \rightarrow \{1, \ldots, r\}, $ such that every hyperedge of $ H $ contains a vertex of each color $1, \dots, r$.
    Additionally, $ p(k, r) $ is
    the minimum number of hyperedges in a $ k $-uniform hypergraph $ H $ such that $ H $ has no panchromatic $ r $-coloring. 
Note that every hyperedge of $ H $ must have size at least $ r $ in order to admit a panchromatic $ r $-coloring.

The following theorem of Kostochka \cite{Kostochka20021} shows that panchromatic hypergraph coloring is closely related to list coloring of multipartite graphs. 
\begin{theorem}\label{prop:ert}
    Let $ N(k, r) $ be the minimum number of vertices in an $ r $-partite graph that is not $ k $-choosable. Then, for any $ r \geq 2 $ and $ k \geq 2, $
    \[ p(k, r) \leq N(k, r) \leq rp(k, r). \]
\end{theorem}

Theorem 1.1 in Cherkashin \cite{CHERKASHIN2018652} showed via an extension of Erd\H{o}s' \cite{Erdős1964445} proof for $ r = 2 $ the following upper bound on $ p(k, r). $

\begin{theorem}\label{thm:pkr}
    For any $ r \geq 2 $ and $ k \geq 2, $ there exists some $ c > 0 $ such that
    \[ p(k, r) \leq c \frac{k^2 \ln r}{r} \left( \frac{r}{r-1} \right)^k. \]
\end{theorem}

From Theorems \ref{prop:ert} and \ref{thm:pkr}, we get a lower bound on the choosability of $ K_{n \star r}. $

\begin{corollary}\label{cor:ch}
    If $ n \geq c \frac{k^2 \ln r}{r} \left( \frac{r}{r-1} \right)^k$, then $ \ch(K_{n \star r}) > k$. That is, $ N(k, r) \leq ck^2 \ln r \left( \frac{r}{r-1} \right)^k. $
\end{corollary}


Given the connection between  list coloring and hypergraph coloring, it is natural to find a connection between paintability and some form of online hypergraph coloring. Aslam and Dhagat \cite{ASLAM1993355} introduced online 2-coloring of $ k $-uniform hypergraphs and Duraj et al. \cite{duraj2015chip} demonstrated its connection to the paintability of $ K_{n,n}$. Later, Khuzieva et al.~\cite{khuzieva2017} extended online 2-coloring to online panchromatic $ r $-coloring, as follows.

    An \emph{online panchromatic $ r $-coloring} of a $ k $-uniform hypergraph with $ n $ hyperedges is defined 
    in terms of a game with a Presenter and a Colorer.
    Throughout the game, the players builds a vertex-colored hypergraph $\mathcal H$ one vertex at a time.
    The Presenter and Colorer take turns.
    On the Presenter's turn, the Presenter introduces a new vertex $ v $ and indicates to 
    which hyperedges of $\mathcal H$ $v$ belongs. 
    Presenter may not add $v$ to a hyperedge of $\mathcal H$ already containing $k$ vertices. On the Colorer's turn, Colorer gives $ v $ one of $ r $ colors.
    The game ends when $\mathcal H$ contains $n$ hyperedges, each with exactly $k$ vertices.
    
    At the end of the game, the Presenter wins if there exists a hyperedge with fewer than $ r $ colors.
    The Colorer wins if the hypergraph $\mathcal H$ has a panchromatic coloring.
    We write $ p_{OL}(k, r) $ for the minimum value $n$ such that the Presenter wins in an online panchromatic $ r $-coloring of a $ k $-uniform hypergraph with $n$ hyperedges.

\subsection{A symmetric chip game}

Aslam and Dhagat \cite{AslamDhagat} defined a chip game related to the online proper $r$-coloring of a $k$-hypergraph, in which Colorer's goal is simply to avoid creating a monochromatic edge. Similar to the chip game defined above, there are a Pusher and a Chooser who alternate turns. 
When $r = 2$, the chip game of Aslam and Dhagat is equivalent to the game of Duraj, Gutowski, and Kozik \cite{duraj2015chip} with two columns.

In order to compute small values of $ p_{OL} (k, r), $ we introduce a symmetric variant of the chip game of Duraj et al.~\cite{duraj2015chip}.
Akhmejanova, Bogdanov, and Chelnokov~\cite{Akh} also show that the values $p_{OL}(k,r)$ can be represented using a chip game.

\begin{definition}
    A \emph{symmetric} $ (k, n \star r) $ chip game is a $ (k, n \star r) $ chip game with the following additional constraint for Pusher.
    \begin{quote}
    At the beginning of the game, we label the chips in each column from $ 1 $ to $ n. $ On each turn, the Pusher chooses a subset $ T \subseteq \{1, \dots, n\} $ and pushes exactly those chips with a label in $ T. $
    \end{quote}
\end{definition}

%

We have the following connection between online panchromatic coloring of uniform hypergraphs and the chip game.

\begin{theorem}\label{thm:conn}
   Pusher has a winning strategy in the symmetric $(k,n \star r)$ chip game if and only if $p_{OL}(k,r) \leq n$.
\end{theorem}
\begin{proof}
  Suppose first that 
  Pusher has no winning strategy in the $ (k,n \star r) $ chip game. 
  We show that  Colorer has a winning strategy in the online $k$-uniform hypergraph coloring game played with $r$ edges and $k$ colors, so that $p_{OL}(k,r) > n$.
    In the hypergraph coloring game, we call our hypergraph $ H = (V, E)$, and we write $ E = \{e_1, \ldots, e_n\}$.
    We also consider a symmetric $(k, n \star r)$ chip game. In each column $i$, label the chips $ c_{i_1}, \ldots, c_{i_n}. $.
    Colorer fixes a winning strategy $\Sigma$ for Remover in the symmetric $(k, n \star r)$ chip game and proceeds as follows.

    On each turn, the Presenter presents a vertex $ v $ and hyperedges $ e_j $ for $ j \in J $ to which $v$ belongs.
    Colorer interpret's Presenter's move as a move in the symmetric chip game in which Pusher pushes the corresponding chips $ c_{j} $ in each column $i$ for each  $ j \in J$.
    If Remover's strategy $\Sigma$ removes the $ i $th column, then Colorer colors $ v $ with the color $ i. $ Since the Remover has a winning strategy, all of the chips will be removed before reaching the $ k $th row without being removed. 
    In other words, the chips with label $c_{j}$ are removed
    from each column $i \in \{1, \dots, r\}$
    after being pushed at most $k$ times.
    Correspondingly, each hyperedge $ e_j $ is colored with each color $i \in \{1, \dots, r\}$ after being presented at most $ k $ times. 
     Thus, the Colorer has a winning strategy.

    On the other hand, suppose that Pusher has a winning strategy in the symmetric $ (k, n \star r) $ chip game. We show that Presenter has a winning strategy in the hypergraph pancoloring game on a $k$-uniform hypergraph with $n$ edges and $r$ colors.
    Again, in each column,  label the chips $ c_{1}, \ldots, c_{n}. $ 
    Suppose that the hypergraph coloring game is played with a $ k $-uniform hypergraph $ H = (V, E) $ with $ E = \{c_{1}, \ldots, c_{n}\}. $ Presenter fixes a winning strategy of Pusher in the symmetric $(k,n \star r)$ and proceeds as follows.

    If the winning strategy for the Pusher pushes a set $ V = \{c_{j}: j \in J\} $ of chips, then the Presenter presents a vertex $ v $ belonging to the edge set $\{e_j: j \in J\}$. If the Colorer colors $ v $ with the color $ i $, then the Presenter 
    imagines that Remover removes chips in the $ i $th column. Since the Pusher has a winning strategy, there will be a chip in some column, say column $ i $, which reaches the $ k $th row.
    Correspondingly, some edge $e_i$
    is presented $ k $ times and none of its $ k $ vertices are colored $ i. $ Thus $e_j $ is not panchromatically colored, so the Presenter wins.
\end{proof}



Aslam and Dhagat used their chip game to show that $p_{OL}(n,2) \geq 2^{n-1}$, and an application of their same method
to the symmetric chip game with $r$ columns
shows that 
$p_{OL}(n,r) \geq \left ( \frac{r}{r-1} \right )^{n-1} $
for all $r \geq 2$.
For $r = 2$, this lower bound is within a constant factor of being best possible. Indeed, Duraj et al.~\cite{duraj2015chip} show that Pusher has a winning strategy in the $(k,8 \times 2^k, 8 \times 2^k)$ chip game. By making one copy of each chip for each column, the same strategy shows that Pusher wins the symmetric $(k,16 \times 2^k, 16 \times 2^k)$ chip game.
Therefore, Theorem \ref{thm:conn} implies that  $ p_{OL}(n,2) \leq 16 \times 2^n$.
 For $r \geq 3$, the following bound of Khuzieva et al.~is best known:
\begin{theorem}[\cite{khuzieva2017}]
    If $n > r$, then $p_{OL}(n,r)  \leq 3r(r-1)^2 n \left ( \frac{r}{r-1} \right )^{n+1}$.
\end{theorem}

\begin{corollary}
\label{cor:pOL}
        If $ \chi_P(K_{n \star r}) \leq k, $ then $ p_{OL}(k, r) > n, $ and if $ \chi_P(K_{n \star r}) > k, $ then $ p_{OL}(k, r) \leq rn. $
\end{corollary}
\begin{proof}
    If $\chi_P(K_{n \star r}) \leq k$, then by Theorem \ref{thm:chippaint}, Remover has a winning strategy in the $(k,n\star r)$ chip game. In particular, Remover has a winning strategy in the symmetric $(k,n\star r)$ chip game, so by Theorem \ref{thm:conn}, $p_{OL}(k,r)>n$. 

    On the other hand, if $\chi_P(K_{n \star r}) > k$, then by Theorem \ref{thm:chippaint}, Pusher has a winning strategy in the $(k,n \star r)$ chip game. By making one copy of each chip per column, Pusher therefore has a winning strategy in the symmetric $(k,rn \star r)$ chip game, so by Theorem \ref{thm:conn}, $p_{OL}(k,r) \leq rn$.
\end{proof}

\subsection{An improved asymptotic bound for $p_{OL}(n,m)$}
We now use the relationship between chip games and panchromatic hypergraph coloring to obtain a lower bound for $\chi_{P}(K_{n\star m})$ which best known for large $n$ and fixed $m \geq 3$. Our result also implies an upper bound on $p_{OL}(n,m)$ for large $n$ and fixed $m \geq 3$ as a corollary, and
    this corollary
    improves a previous bound of Khuzieva et al.~\cite{khuzieva2017}.
In order to prove our lower bound, we use the chip game. Our method uses ideas from Khuzieva et al.~\cite{khuzieva2017} but makes several improvements.

We consider the chip game with $m \geq 2$ fixed columns.
We relabel the rows as $0,\dots,k$, this time with row $k$ as the initial row and with row $0$ as the target row. When a chip in row $i$ is pushed, it moves to row $i-1$.
We define a \emph{brick} to be a set of $f(r)$ chips 
in row $r$ of the same column, where $f$ is a function that we define below.
We define a brick at row $0$ to consist of a single chip, so that $f(0) = 1$.
For $0 \leq r\leq k$, we let $g(r)=\left \lceil \frac{f(r)}{m-1} \right \rceil $ for $r \geq 0$.
We call a set 
of $g(r)$ chips at row $r$ a $\frac{1}{m-1}$ fraction of a brick.

We define recursively the number of chips in a brick at row $r$ with the recurrence 
\[
f(r)=f(r-1)+g(r-1).
\]
We make the following observation about the function $f$.
\begin{claim}\label{claim:upperbound}
For $k \geq 0$,
   \[
   f(k)\leq \sum_{j=0}^{k}\left(\frac{m}{m-1}\right)^{j} = \frac{\left(\frac{m}{m-1}\right)^{k+1}-1}{\frac{m}{m-1}-1}
   <m\left(\frac{m}{m-1}\right)^{k}.
    \]
\end{claim}

\begin{proof}
    The proof goes by induction. The base case $k=0$ is clearly true. The induction step is as follows. For $k \geq 1$, 
\begin{eqnarray*}
    f(k)&=&f(k-1)+g(k-1) 
    = f(k-1)+\left \lceil \frac{f(k-1)}{m-1} \right \rceil \\
    &\leq& f(k-1)+\frac{f(k-1)}{m-1}+1=\frac{m}{m-1}f(k-1)+1\\
    &\leq& \sum_{j=0}^{k-1}\left(\frac{m}{m-1}\right)^{j+1}+1=\sum_{j=0}^{k}\left(\frac{m}{m-1}\right)^{j}.
\end{eqnarray*}
\end{proof}

Now, we prove a lower bound on the paintability of the complete multipartite graph $K_{\left(m(k+1)f(k)\right)\star m}$, which ultimately implies an upper bound on $p_{OL}(n,m)$.
\begin{theorem}
\label{thm:brick}
For all integers $m,k \geq 2$
\[
\chi_{P}(K_{\left(m(k+1)f(k)\right)\star m}) >  k.
\]
\end{theorem}
\begin{proof}
    By Theorem \ref{thm:chippaint}, it is sufficient to show that Pusher wins the chip game with $m$ columns, in which each column starts with $m(k+1)f(k)$ tokens. We show that Pusher wins by using the following strategy.
    On each turn, 
    for each column $C$, Pusher selects the brick $B$ in $C$ whose row $r$ is minimum (that is, $B$ is farthest from the starting row) and pushes all chips in $B$.
    When Pusher pushes $B$, the pushed chips form a full brick plus a $\frac{1}{m-1}$ of a brick in row $r-1$, since
    $f(k)=f(k-1)+g(k-1)$. 

    If at any point there exist $m-1$ copies of a $\frac{1}{m-1}$-brick occupying the same space, we merge them back into one full brick. This action is possible, since $(m-1)\lceil \frac{f(k)}{m-1} \rceil \geq f(k)$. Any extra chips produced in the process as a result of the inequality are ignored, which changes neither the number of bricks nor the integrity of the proof.

    We claim that after a Pusher and Remover turn, the number of bricks does not decrease.
    Indeed, on Pusher's turn, Pusher advances exactly $m$ bricks, and exactly one of these bricks is deleted. After Remover's turn is over, as the chips of the surviving $m-1$ bricks have advanced one row, these surviving $m-1$ bricks become $(m-1) \cdot \frac{m}{m-1} = m$ bricks.
    We make the following claim about the positions of the bricks throughout the game.

    \begin{claim}
    \label{claim:brick}
        At the end of each round, in each column $C$,
        the highest row $r_C$ occupied by a full brick has at most $2$ bricks, unless $r_C = k$. Furthermore, each row strictly between row $k$ and $r_C$ has at most $1$ brick.
    \end{claim}
    \begin{proof}[Proof of claim]
    We induct on the number of rounds that have already been played. When no round has been played, the claim clearly holds.


    Now, for our induction case, assume that at the end of a round, we have a state in which the claim holds. Then, 
    suppose that in column $C$, Pusher moves a brick from row $i$ to row $i-1$. If this brick is removed on Remover's turn, then the claim clearly holds for column $C$. Otherwise, we know that before Pusher's turn, there must not be any bricks in column $C$ above row $i$, 
    as otherwise Pusher would not choose to push the brick in row $i$.
    Therefore,  before Pusher's turn,
    row $i-1$ has at most $\frac{m-2}{m-1}$ bricks. Since the brick pushed from row $i$ splits into a full brick and a $\frac{1}{m-1}$ of a brick on row $i-1$, 
    row $i-1$ has at most $2$ bricks after Pusher's turn, and no row above $i-1$ has a full brick.
     In order to ensure that every other row between row $k$ and row $i-1$ has at most $1$ brick, we only need to consider row $i$, since each row below it is left untouched. Since row $i$ starts with at most $2$ bricks, and $1$ brick is moved, this means row $i$ ends with at most $1$ brick after Pusher's turn. Thus, the claim holds in every case.
    \end{proof}

    If Pusher is unable to follow the strategy outlined above, then this means Pusher is unable to find a brick to push some column $C$. Thus, assume that one column $C$ runs out of full bricks and that Pusher has not won.
     Since Pusher pushes a brick from row $k$ of each column on the first move, row $k$ of each column other than $C$ has at most $m(k+1)-1$ bricks, and row $k$ of column $C$ has at most $\frac{m-2}{m-1}$ bricks.
    By Claim \ref{claim:brick}, above row $k$, each column other than $C$ has at most $k$ bricks, and $C$ has at most $\frac{m-2}{m-1}(k-1)$ bricks above row $k$.
    Therefore, the total number of bricks on the board is at most 
    \[(m-1)(m(k+1)-1) + (m-1)k + k\left (\frac{m-2}{m-1}  \right ) < (m-1)(m(k+1) - 1) + mk \leq m^2 (k+1), \]
    contradicting the observation that the number of bricks on the board after each round is at least $m^2(k+1)$.
    Therefore, while Pusher has not yet won, Pusher has a legal move following the strategy described above, which implies that the game ends with Pusher winning.
\end{proof}

    In the proof of Theorem \ref{thm:brick}, Claim \ref{claim:brick} gives us an upper bound on the maximum number of bricks that can be on the board at any given point. 
    We note that Khuzieva et al.~\cite{khuzieva2017}
    have a similar claim, but they only show that each row other than row $k$ has at most three bricks. 
   Our improved Claim \ref{claim:brick} ultimately gives an improved result.

We also get a bound on the online panchromatic number, which improves Proposition 3 of \cite{khuzieva2017} by a factor of roughly $3$ when $m \geq 3$ and $k$ is large.

\begin{corollary}
For $m \geq 2$,
    \[ p_{OL}(k, m) \leq m^{3}(k+1)
    \left(\frac{m}{m-1}\right)^k. \]
\end{corollary}
\begin{proof}
    By Theorem \ref{thm:brick}, 
    $\chi_P(K_{n \star r}) > k$ 
    when $n = m(k+1)f(k)$. As $f(k) < m\left (\frac{m}{m-1} \right )^k$ by Claim \ref{claim:upperbound},
     $\chi_P(K_{n \star r}) > k$  when $n = (k+1) m^2 \left (\frac{m}{m-1} \right )^k$.
    Then, by Corollary \ref{cor:pOL}, 
    $p_{OL}(k, m) \leq m^{3}(k+1)
    \left(\frac{m}{m-1}\right)^k.$
\end{proof}

\section{Conclusion}
\label{sec:conclusion}
Finding the exact paintability of graph is a difficult problem, because it involves a game between two players. To determine the exact paintability of a graph, we not only need to find an optimal strategy for the Painter, and also find a optimal strategy for the Lister. For this reason, only the paintabilities of some graphs with simple structures, such as certain complete multipartite graphs, are known.
Can we determine the paintability of other complete multipartite graphs with small parts and other structures? How large is the gap between choosability and paintability for general graphs? These are questions to be solved.

 We were able to obtain exact values for  $\chi_P(K_{3\star 4})$, $\chi_P(K_{3\star5})$, and $\chi_P(K_{3\star 6})$, but we are still interested in finding the exact value of $\chi_P(K_{3\star7})$.

\begin{question}
What is $\chi_P(K_{3 \star 7})$?
\end{question}


Kozik, Micek, and Zhu~\cite{KMZ} asked for the size of the gap $\chi_P(K_{3\star n})-ch(K_{3\star n})$, and we have shown that for $n \leq 6$, this gap is zero. However, the following questions remain open.

\begin{question}
   Is it true that $\ch(K_{3 \star n}) = \chi_P(K_{3 \star n})$ for each positive integer $n$?
\end{question}
Another natural problem is whether the method of Duraj \cite{duraj2015chip} for the two-column chip game can be extended to games with more columns. We pose the following specific question:
\begin{question}
\label{q:r-1r}
    For fixed $r \geq 3$, is $\chi_P(K_{r \star n}) = \log_{\left (\frac{r}{r-1} \right )} n + O(1)$?
\end{question}

Question \ref{q:r-1r} is related to the following question about panchromatic hypergraph coloring.

 \begin{question}
    For each $r \geq 2$, what is the asymptotic growth rate of $\frac{p(n,r)}{p_{OL}(n,r)}$ as $n$ increases?
\end{question}

Finally, we developed computational methods that determine the paintabilities specifically of complete multipartite graphs. This leads to the following natural question.
\begin{question}
For which other graph classes can paintabilities be computationally determined?
\end{question}

\bibliographystyle{plain}
\bibliography{citation} 
\end{document}